\documentclass[12pt,a4paper,reqno]{amsart}
\allowdisplaybreaks
\usepackage{amsmath}
\usepackage{amsfonts}
\usepackage{amssymb,amsthm,amsfonts,amsthm,latexsym,enumerate,url,cases}
\numberwithin{equation}{section}
\usepackage{mathrsfs}
\usepackage{hyperref}
\hypersetup{colorlinks=true,citecolor=blue,linkcolor=blue,urlcolor=blue}
\def\pmod #1{\ ({\rm{mod}}\ #1)}

\theoremstyle{plain}
\newtheorem{theorem}{Theorem}
\newtheorem{lemma}{Lemma}

\newtheorem{corollary}{Corollary}
\newtheorem{proposition}{Proposition}
\newtheorem{conjecture}{Conjecture}
\theoremstyle{definition}

\usepackage{etoolbox}
\makeatletter
\patchcmd{\@settitle}{\uppercasenonmath\@title}{}{}{}
\patchcmd{\@setauthors}{\MakeUppercase}{}{}{}
\patchcmd{\section}{\scshape}{}{}{}
\makeatother

\begin{document}

\title
[{On a conjecture of Ram\'{\i}rez Alfons\'{\i}n and Ska{\l}ba II}]
{On a conjecture of Ram\'{\i}rez Alfons\'{\i}n and Ska{\l}ba II}

\author
[Y. Ding, \quad W. Zhai \quad {\it and} \quad  L. Zhao] 
{Yuchen Ding, \quad Wenguang Zhai \quad {\it and} \quad Lilu Zhao}

\address{(Yuchen Ding) School of Mathematical Sciences,  Yangzhou University, Yangzhou 225002, People's Republic of China}
\email{ycding@yzu.edu.cn}
\address{(Wenguang Zhai) Department of Mathematics,  China University of Mining and Technology, Beijing 100083, People's Republic of China}
\email{zhaiwg@hotmail.com}
\address{(Lilu Zhao) School of Mathematics, Shandong University, Jinan 250100, People's Republic of China}
\email{zhaolilu@sdu.edu.cn}


\begin{abstract}
Let $1<c<d$ be two relatively prime integers and $g_{c,d}=cd-c-d$.  We confirm, by employing the Hardy--Littlewood method, a 2020 conjecture of Ram\'{\i}rez Alfons\'{\i}n and Ska{\l}ba which states that
$$\#\left\{p\le g_{c,d}:p\in \mathcal{P}, ~p=cx+dy,~x,y\in \mathbb{Z}_{\geqslant0}\right\}\sim \frac{1}{2}\pi\left(g_{c,d}\right) \quad (\text{as}~c\rightarrow\infty),$$
where $\mathcal{P}$ is the set of primes, $\mathbb{Z}_{\geqslant0}$ is the set of nonnegative integers and $\pi(t)$ denotes the number of primes not exceeding $t$. 
\end{abstract}
\maketitle

\section{Introduction}
Let $1< c<d$ be two relatively prime integers and $g_{c,d}=cd-c-d$.  As early as 1882, Sylvester \cite{Sylvester} showed that $g_{c,d}$ is the largest integer which cannot be represented as the form 
$cx+dy~(x,y\in \mathbb{Z}_{\geqslant0}).$
Furthermore, he proved that for any $0\le s\le g_{c,d}$, exactly one of $s$ and $g_{c,d}-s$ can be written as the form 
$cx+dy~(x,y\in \mathbb{Z}_{\geqslant0})$.
As an immediate consequence, we know that exactly half of the integers between the interval $[0,g_{c,d}]$ can be written as the desired form. Actually, Sylvester's results are the first nontrivial case of the diophantine Frobenius problem \cite{RA}, which asks the largest integer $g_{c_1,...,c_n}$ not of the form
$$c_1x_1+\cdots+c_nx_n \quad (x_1,...,x_n\in \mathbb{Z}_{\geqslant0}),$$
provided that $c_1,...,c_n$ are positive integers with $\gcd(c_1,...,c_n)=1$. There are a huge number of literatures related to the diophantine Frobenius problem. For some of these results, see e.g. the excellent monograph \cite{RA} of Ram\'{\i}rez Alfons\'{\i}n.

Motivated by Sylvester's theorems, Ram\'{\i}rez Alfons\'{\i}n and Ska{\l}ba \cite{RS} considered the diophantine Frobenius problem in primes. Precisely, let 
$\pi_{c,d}$ be the number of primes not exceeding $g_{c,d}$ with the form 
$cx+dy~(x,y\in \mathbb{Z}_{\geqslant0})$. By a very enlightening argument, Ram\'{\i}rez Alfons\'{\i}n and Ska{\l}ba proved that for any $\varepsilon>0$, there is a constant $k(\varepsilon)>0$ such that
$$\pi_{c,d}\geqslant k(\varepsilon)\frac{g_{c,d}}{(\log g_{c,d})^{2+\varepsilon}}.$$
On observing the antisymmetry property of the integers with the form  $cx+dy~(x,y\in \mathbb{Z}_{\geqslant0})$ found by Sylvester, they naturally posed the following conjecture. 

\begin{conjecture}[Ram\'{\i}rez Alfons\'{\i}n and Ska{\l}ba]\label{conjecture1}
Let $1< c<d$ be two relatively prime integers, then $$\pi_{c,d}\sim\frac{\pi(g_{c,d})}{2}\quad (\text{as}~c\rightarrow\infty),$$
where $\pi(t)$ is the number of primes up to $t$.
\end{conjecture}

They gave some remarks below Conjecture \ref{conjecture1}. `{\it In the same spirit as the prime number theorem, this conjecture seems to be out of reach.}' Also, they mentioned that this conjecture has some difficulties sharing the same flavor of Linnik's problem related the minimal primes in arithmetic progressions.
Recently, the first named author made some progress on Conjecture \ref{conjecture1}. 
For real number $N\geqslant2$, let $1<c<d$ be two relatively prime integers satisfying $cd\leqslant N$. The first named author \cite{Ding} proved that for all but at most $$O\left(N(\log N)^{1/2}(\log\log N)^{1/2+\varepsilon}\right)$$ pairs $c$ and $d$, we have
\begin{align*}
\pi_{c,d}=\frac{\pi(g_{c,d})}{2}+O\left(\frac{\pi(g_{c,d})}{(\log\log (cd))^{\varepsilon}}\right).
\end{align*}
Since $$\frac{\pi(g_{c,d})}{2}+O\left(\frac{\pi(g_{c,d})}{(\log\log (cd))^{\varepsilon}}\right)\sim\frac{\pi(g_{c,d})}{2} \quad (\text{as}~c\rightarrow\infty)$$
and the total number of the relatively prime pairs $c,d$ with $1<c<d$ and $cd\leqslant N$ is $\gg N\log N$. Thus, the first named author actually showed that Conjecture \ref{conjecture1} is true for almost all $c$ and $d$.

It is, however, rather surprising that the complete proof of Conjecture \ref{conjecture1} follows from an application of the classical Hardy--Littlewood method. Perhaps, a novel point in our argument is that only the first coefficient of the `singular series' contributes the main term of the asymptotic formula comparing with the usual applications of the Hardy--Littlewood method. The idea of the proof presented here is in the same spirit as the one developed by a very recent article of the third named author and Chen, Yang \cite{CYZ}.

Now, let's record our result as the following theorem.

\begin{theorem}\label{thm1}
Suppose that $d>c$ are two relatively prime integers with $c$ sufficiently large, then we have
$$\pi_{c,d}\sim \frac12\pi(g_{c,d}), \quad \text{as~}c\rightarrow\infty.$$
\end{theorem}

As usual, we shall firstly investigate the following weighted form related to Conjecture \ref{conjecture1}, i.e.,
\begin{align*}
\psi_{c,d} = \sum_{\substack{n\le g\\ n=cx+dy\\ x,y\in \mathbb{Z}_{\geqslant0}}}\Lambda (n),
\end{align*}
where the von Mangoldt function $\Lambda(n)$ is defined to be
\begin{equation*}
\Lambda(n)=
\begin{cases}
\log p, & \text{if } n=p^{\alpha}~(\alpha>0);\\
0, & \text{otherwise}.
\end{cases}
\end{equation*}

Theorem \ref{conjecture1} will be proved via the following weighted formula by a fairly standard transition. 

\begin{theorem}\label{thm2}
Suppose that $d>c$ are two relatively prime integers with $c$ sufficiently large, then we have
$$\psi_{c,d}\sim \frac{g_{c,d}}{2}, \quad \text{as~}c\rightarrow\infty.$$
\end{theorem}

As an incident product of Theorem \ref{thm2}, we have the following corollary which seems to be of some interests.

\begin{corollary}\label{corollary1}
Suppose that $d>c$ are two relatively prime integers with $c$ sufficiently large, then we have
$$\sum_{\substack{y\leqslant c\\(y,c)=1}}\psi(dy;c,dy)\sim\frac{g_{c,d}}{2}, \quad \text{as~}c\rightarrow\infty,$$
where $$\psi(N;q,m)=\sum_{\substack{n\leqslant N\\ n\equiv m\!\!\pmod{q}}}\Lambda(n).$$ 
\end{corollary}

\section{Outline of the proof: an application of the Hardy--Littlewood method}
We first fix some basic notations to be used frequently. From now on, we write $g$ instead of $g_{c,d}$ for brevity and $c$ is supposed to be sufficiently large. Let $Q$ denote a positive integer depending only on $g$ which shall be decided later. The function $e(t)$ is used to denote $e^{2\pi it}$ as usual. Define the major arcs to be

\begin{align}\label{eq2-1}
\mathfrak{M}(Q)=\bigcup_{1\le q\le Q}\bigcup_{\substack{1\le a\le q\\ (a,q)=1}}\left\{\alpha:\left|\alpha-\frac{a}{q}\right|\le \frac{Q}{qg}\right\}
\end{align}
We make a further provision that $Q< (g/2)^{1/3}$ so that the above subsets are pairwise disjoint. In fact, suppose that
$$\left\{\alpha:\left|\alpha-\frac{a}{q}\right|\le \frac{Q}{qg}\right\}\bigcap\left\{\alpha:\left|\alpha-\frac{a'}{q'}\right|\le \frac{Q}{qg}\right\}\neq \emptyset$$
for some $\frac{a}{q}\neq \frac{a'}{q'}$, then
$$\frac{2Q}{g}\ge \frac{2Q}{qg}\ge \left|\frac{a}{q}-\frac{a'}{q'}\right|\ge \frac{1}{Q^2},$$
which is certainly a contradiction with the provision that $Q< (g/2)^{1/3}$. In addition, we note that
$$\mathfrak{M}(Q)\subseteq \left[\frac{1}{Q}-\frac{Q}{qg}, 1+\frac{Q}{qg}\right] \subseteq\left[\frac{Q}{g}, 1+\frac{Q}{g}\right].$$
We can now define the minor arcs to be
\begin{align}\label{eq2-2}
\mathfrak{m}(Q)=\left[\frac{Q+1}{g}, 1+\frac{Q+1}{g}\right]\setminus \mathfrak{M}(Q).
\end{align}
For any real $\alpha$, let
$$f(\alpha)=\sum_{0\le n\le g}\Lambda(n)e(\alpha n) \quad \text{and} \quad h(\alpha)=\sum_{\substack{0\le x\le d\\ 0\le y\le c}}e(\alpha (cx+dy))$$
By the orthogonality relation, it is clear that
\begin{align}\label{eq2-3}
\psi_{c,d} = \int_{0}^{1}f(\alpha)h(-\alpha)d\alpha=\int_{\mathfrak{M}(Q)}f(\alpha)h(-\alpha)d\alpha+\int_{\mathfrak{m}(Q)}f(\alpha)h(-\alpha)d\alpha.
\end{align}

The remaining parts of our paper will be organized as follows: In the next section, we shall give a suitable bound of the integral on the minor arcs which shows that it contributes the error term of Theorem \ref{thm2}. 
For the integral on the major arcs, we have
\begin{align}\label{eq2-4}
\int_{\mathfrak{M}(Q)}f(\alpha)h(-\alpha)d\alpha&=\sum_{1\le q\le Q}\sum_{\substack{1\le a\le q\\ (a,q)=1}}\int_{\frac{a}{q}-\frac{Q}{qg}}^{\frac{a}{q}+\frac{Q}{qg}}f(\alpha)h(-\alpha)d\alpha \nonumber\\
&=\sum_{1\le q\le Q}\sum_{\substack{1\le a\le q\\ (a,q)=1}}\int_{|\theta|\le \frac{Q}{qg}}f\left(\frac{a}{q}+\theta\right)h\left(-\frac{a}{q}-\theta\right)d\theta.
\end{align}
We would prove, in section 4, that the integral for $q=1$ in Eq. (\ref{eq2-2}) contributes the main term of Theorem \ref{thm2} and the integrals for $2\le q\le Q$ only contribute the error terms. In section 5, we will prove our theorems and corollary.

\section{Estimates of the minor arcs}
The aim of this section is to prove the following proposition.
\begin{proposition}\label{proposition1} For estimates of the minor arcs, we have
\begin{align*}
\int_{\mathfrak{m}(Q)}f(\alpha)h(-\alpha)d\alpha\ll \frac{g(\log g)^6}{Q^{1/2}}+g^{4/5}(\log g)^6.
\end{align*}
\end{proposition}
We need two lemmas listed below. 

\begin{lemma}\label{lem1} We have
\begin{align*}
\sup_{\alpha\in \mathfrak{m}(Q)}\left|f(\alpha)\right|\ll \frac{g(\log g)^4}{Q^{1/2}}+g^{4/5}(\log g)^4.
\end{align*}
\end{lemma}
\begin{proof}
By the Dirichlet approximation theorem (see e.g. \cite[Lemma 2.1]{Vaughan}), there exist $a\in \mathbb{Z}$ and $q\in \mathbb{Z}^+$ such that
$$(a,q)=1, \quad 1\le q\le \frac{g}{Q} \quad \text{and} \quad \left|\alpha-\frac{a}{q}\right|\le \frac{Q}{qg}$$
for any $m\in\mathfrak{m}(Q)$. First of all, we show that $q>Q$ for these $\alpha \in \mathfrak{m}(Q)$. Suppose the contrary, i.e., $q\le Q$, then from $\left|\alpha-\frac{a}{q}\right|\le \frac{Q}{qg}$
we know that 
$$\alpha\le \frac{a}{q}+\frac{Q}{qg}\le \frac{Q}{qg}\le \frac{Q}{g}<\frac{Q+1}{g} \quad  (\text{if } a\le 0)$$
and 
$$\alpha \ge \frac{q+1}{q}-\frac{Q}{qg}=1+\frac{1}{q}\left(1-\frac{Q}{g}\right)\ge 1+\frac{1}{2Q}>1+\frac{Q+1}{g} \quad  (\text{if } a\ge q+1),$$
which contradict with Eq. (\ref{eq2-2}). We are left over to consider the case that $1\le a\le q\le Q$. In this case, we have $\alpha\in \mathfrak{M}(Q)$ by the definition of the major arcs which is still a contradiction. Thus, we have proved that $q>Q$ for these $\alpha \in \mathfrak{m}(Q)$. 
We are in a position to introduce a fairly remarkable theorem of Vinogradov (see e.g. \cite[Theorem 3.1]{Vaughan}) which states that
for $(a,q)=1$ with $1\le q\le g$ and $\left|\alpha-\frac{a}{q}\right|\le \frac{1}{q^2}$, we have
\begin{align*}
f(\alpha)\ll \left(\frac{g}{q^{1/2}}+q^{4/5}+g^{1/2}q^{1/2}\right)(\log g)^4.
\end{align*}
Employing this estimate, we deduce from $Q<q\le \frac{g}{Q}$ that 
$$\sup_{\alpha\in \mathfrak{m}(Q)}\left|f(\alpha)\right|\ll \left(\frac{g}{Q^{1/2}}+g^{4/5}+g^{1/2}(g/Q)^{1/2}\right)(\log g)^4\ll \frac{g(\log g)^4}{Q^{1/2}}+g^{4/5}(\log g)^4.$$

This completes the proof of Lemma \ref{lem1}
\end{proof}

\begin{lemma}\label{lem2} We have
$$\int_{0}^{1}|h(-\alpha)|d\alpha\ll (\log g)^2.$$
\end{lemma}
\begin{proof}
Recall that for any real number $\alpha$ and integers $N_1<N_2$, we have
$$\sum_{n=N_1+1}^{N_2}e(\alpha n)\ll \min\{N_2-N_1, \Vert\alpha\Vert^{-1}\}$$
(see e.g. \cite[Lemma 4.7]{Nathanson}), where 
$\Vert\alpha\Vert=\min\{|\alpha-n|:n\in \mathbb{Z}\},$
from which it follows that
\begin{align*}
h(-\alpha)=\sum_{x\le d}e(-\alpha cx)\sum_{y\le c}e(-\alpha dy)\ll \min\{d, \Vert c\alpha\Vert^{-1}\}\min\{c, \Vert d\alpha\Vert^{-1}\}.
\end{align*}
The rest of the proof will be devoted to derive the following `elementary estimates' 
\begin{align}\label{eq3-1}
\int_{0}^{1}\min\{d, \Vert c\alpha\Vert^{-1}\}\min\{c, \Vert d\alpha\Vert^{-1}\}d\alpha\ll (\log g)^2,
\end{align}
which constitutes one of the the main workhorse of the whole article. We make the observation that it is equivalent to prove that
$$\int_{0}^{1/2}\min\{d, \Vert c\alpha\Vert^{-1}\}\min\{c, \Vert d\alpha\Vert^{-1}\}d\alpha\ll (\log g)^2.$$
For $0\le \alpha \le \frac{1}{cd}$, we have the following trivial estimates that
$$\int_{0}^{1/(cd)}\min\{d, \Vert c\alpha\Vert^{-1}\}\min\{c, \Vert d\alpha\Vert^{-1}\}d\alpha\le \int_{0}^{1/(cd)}dc~d\alpha\le 1.$$
For $\frac{1}{cd}\le \alpha\le \frac{1}{2\sqrt{cd}}$, it is plain that
$$\frac{1}{d}\le c\alpha \le \frac{1}{2}\sqrt{\frac{c}{d}}<\frac{1}{2}, \quad \text{i.e.,} \quad \Vert c\alpha\Vert=c\alpha,$$
from which we deduce that
$$\int_{\frac{1}{cd}}^{\frac{1}{2\sqrt{cd}}}\min\{d, \Vert c\alpha\Vert^{-1}\}\min\{c, \Vert d\alpha\Vert^{-1}\}d\alpha\le \int_{\frac{1}{cd}}^{\frac{1}{2\sqrt{cd}}}\frac{1}{c\alpha}c~d\alpha\ll \log (cd)\ll \log g.$$
It remains to prove that 
$$\int_{\frac{1}{2\sqrt{cd}}}^{1/2}\min\{d, \Vert c\alpha\Vert^{-1}\}\min\{c, \Vert d\alpha\Vert^{-1}\}d\alpha\ll (\log g)^2.$$ 
The above interval is contained in the following union of a few disjoint short intervals
$$\left[\frac{1}{2\sqrt{cd}}, \frac{1}{2}\right]\subseteq \bigcup_{\left\lfloor \frac{\sqrt{cd}}{2}\right\rfloor\le \ell\le \left\lfloor \frac{cd}{2}\right\rfloor}\left[\frac{\ell}{cd}-\frac{1}{2cd},\frac{\ell}{cd}+\frac{1}{2cd}\right].$$
By the above inclusion relation, it suffices to show that
$$\sum_{\left\lfloor \frac{\sqrt{cd}}{2}\right\rfloor\le \ell\le \left\lfloor \frac{cd}{2}\right\rfloor}\int_{\frac{\ell}{cd}-\frac{1}{2cd}}^{\frac{\ell}{cd}+\frac{1}{2cd}}\min\{d, \Vert c\alpha\Vert^{-1}\}\min\{c, \Vert d\alpha\Vert^{-1}\}d\alpha\ll (\log g)^2.$$
Making changes of the variables $\alpha=\frac{\ell}{cd}+\theta$, it is equivalent to prove that
\begin{align}\label{eq3-2}
\sum_{\left\lfloor \frac{\sqrt{cd}}{2}\right\rfloor\le \ell\le \left\lfloor \frac{cd}{2}\right\rfloor}\int_{-\frac{1}{2cd}}^{\frac{1}{2cd}}\min\left\{d, \left\Vert c\theta+\frac{\ell}{d}\right\Vert^{-1}\right\}\min\left\{c, \left\Vert d\theta+\frac{\ell}{c}\right\Vert^{-1}\right\}d\theta\ll (\log g)^2.
\end{align}
We separate the proof into four cases.

Case I. For $\left\lfloor \frac{\sqrt{cd}}{2}\right\rfloor\le \ell\le \left\lfloor \frac{cd}{2}\right\rfloor$ with $c\nmid \ell$ and $d\nmid \ell$, we have
$$\left\Vert c\theta+\frac{\ell}{d}\right\Vert\ge \frac{1}{2}\left\Vert\frac{\ell}{d}\right\Vert \quad \text{and} \quad \left\Vert d\theta+\frac{\ell}{c}\right\Vert\ge \frac{1}{2}\left\Vert\frac{\ell}{c}\right\Vert$$
since $|\theta|\le \frac{1}{2cd}$. Therefore, we obtain that
\begin{align}\label{eq3-3}
\sum_{\substack{\left\lfloor \frac{\sqrt{cd}}{2}\right\rfloor\le \ell\le \left\lfloor \frac{cd}{2}\right\rfloor\\ c\nmid\ell~\text{and}~d\nmid \ell}}\int_{-\frac{1}{2cd}}^{\frac{1}{2cd}}\min\left\{d, \left\Vert c\theta+\frac{\ell}{d}\right\Vert^{-1}\right\}&\min\left\{c, \left\Vert d\theta+\frac{\ell}{c}\right\Vert^{-1}\right\}d\theta\nonumber\\
&\le 4\sum_{\substack{\left\lfloor \frac{\sqrt{cd}}{2}\right\rfloor\le \ell\le \left\lfloor \frac{cd}{2}\right\rfloor\\ c\nmid\ell~\text{and}~d\nmid \ell}}\int_{-\frac{1}{2cd}}^{\frac{1}{2cd}}\left\Vert\frac{\ell}{d}\right\Vert^{-1}\left\Vert\frac{\ell}{c}\right\Vert^{-1}d\theta\nonumber\\
&=\frac{4}{cd}\sum_{\substack{\left\lfloor \frac{\sqrt{cd}}{2}\right\rfloor\le \ell\le \left\lfloor \frac{cd}{2}\right\rfloor\\ c\nmid\ell~\text{and}~d\nmid \ell}}\left\Vert\frac{\ell}{d}\right\Vert^{-1}\left\Vert\frac{\ell}{c}\right\Vert^{-1}.
\end{align}
Now, by the Euclidean division we can assume that $\ell=ch+r$ with $1\le r\le c-1$. It then follows that
\begin{align*}
\sum_{\substack{\left\lfloor \frac{\sqrt{cd}}{2}\right\rfloor\le \ell\le \left\lfloor \frac{cd}{2}\right\rfloor\\ c\nmid\ell~\text{and}~d\nmid \ell}}\left\Vert\frac{\ell}{d}\right\Vert^{-1}\left\Vert\frac{\ell}{c}\right\Vert^{-1}&\le \sum_{\substack{0\le h\le \frac{d}{2},~1\le r\le c-1\\ h\not\equiv -c^{-1}r\!\pmod{d}}}\left\Vert\frac{ch+r}{d}\right\Vert^{-1}\left\Vert\frac{r}{c}\right\Vert^{-1}\\
&=\sum_{1\le r\le c-1}\left\Vert\frac{r}{c}\right\Vert^{-1}\sum_{0\le h\le \frac{d}{2},~h\not\equiv -c^{-1}r\!\pmod{d}}\left\Vert\frac{ch+r}{d}\right\Vert^{-1},\\
&\le \sum_{1\le r\le c-1}\left\Vert\frac{r}{c}\right\Vert^{-1}\sum_{1\le r'\le d-1}\left\Vert\frac{r'}{d}\right\Vert^{-1},
\end{align*}
where $c^{-1}c\equiv 1\pmod{d}$ and the last inequality follows from the fact that
$$ch+r\not\equiv ch'+r \pmod{d} \quad (\text{for } 0\le h\neq h'\le d/2).$$
It can be seen that
$$\sum_{1\le r\le c-1}\left\Vert\frac{r}{c}\right\Vert^{-1}\le 2\sum_{1\le r\le \lfloor c/2\rfloor}\left\Vert\frac{r}{c}\right\Vert^{-1}=2\sum_{1\le r\le \lfloor c/2\rfloor}\frac{c}{r}\ll c\log c\ll c\log g$$
and
$$\sum_{1\le r'\le d-1}\left\Vert\frac{r'}{d}\right\Vert^{-1}\le 2\sum_{1\le r'\le \lfloor d/2\rfloor}\left\Vert\frac{r'}{d}\right\Vert^{-1}=2\sum_{1\le r'\le \lfloor d/2\rfloor}\frac{d}{r'}\ll d\log d\ll d\log g,$$
from which we deduce from Eq. (\ref{eq3-3}) that
\begin{align*}
\sum_{\substack{\left\lfloor \frac{\sqrt{cd}}{2}\right\rfloor\le \ell\le \left\lfloor \frac{cd}{2}\right\rfloor\\ c\nmid\ell~\text{and}~d\nmid \ell}}\int_{-\frac{1}{2cd}}^{\frac{1}{2cd}}\min\left\{d, \left\Vert c\theta+\frac{\ell}{d}\right\Vert^{-1}\right\}\min\left\{c, \left\Vert d\theta+\frac{\ell}{c}\right\Vert^{-1}\right\}d\theta\ll(\log g)^2.
\end{align*}

Case II. For $\left\lfloor \frac{\sqrt{cd}}{2}\right\rfloor\le \ell\le \left\lfloor \frac{cd}{2}\right\rfloor$ with $c\nmid \ell$ but $d|\ell$, we have
$$\left\Vert d\theta+\frac{\ell}{c}\right\Vert\ge \frac{1}{2}\left\Vert\frac{\ell}{c}\right\Vert$$
and then it follows that
\begin{align*}
\sum_{\substack{\left\lfloor \frac{\sqrt{cd}}{2}\right\rfloor\le \ell\le \left\lfloor \frac{cd}{2}\right\rfloor\\ c\nmid\ell~\text{and}~d| \ell}}\int_{-\frac{1}{2cd}}^{\frac{1}{2cd}}\min\left\{d, \left\Vert c\theta+\frac{\ell}{d}\right\Vert^{-1}\right\}&\min\left\{c, \left\Vert d\theta+\frac{\ell}{c}\right\Vert^{-1}\right\}d\theta\nonumber\\
&\le 2\sum_{\substack{\left\lfloor \frac{\sqrt{cd}}{2}\right\rfloor\le \ell\le \left\lfloor \frac{cd}{2}\right\rfloor\\ c\nmid\ell~\text{and}~d| \ell}}\int_{-\frac{1}{2cd}}^{\frac{1}{2cd}}d\left\Vert\frac{\ell}{c}\right\Vert^{-1}d\theta\nonumber\\
&=\frac{2}{c}\sum_{\substack{\left\lfloor \frac{\sqrt{cd}}{2}\right\rfloor\le \ell\le \left\lfloor \frac{cd}{2}\right\rfloor\\ c\nmid\ell~\text{and}~d| \ell}}\left\Vert\frac{\ell}{c}\right\Vert^{-1}.
\end{align*}
On writing $\ell=d\ell^*$, we find that
$$\sum_{\substack{\left\lfloor \frac{\sqrt{cd}}{2}\right\rfloor\le \ell\le \left\lfloor \frac{cd}{2}\right\rfloor\\ c\nmid\ell~\text{and}~d| \ell}}\left\Vert\frac{\ell}{c}\right\Vert^{-1}\le \sum_{\substack{1\le \ell^*\le \frac{c}{2}\\ c\nmid \ell^*}}\left\Vert\frac{d\ell^*}{c}\right\Vert^{-1}\le \sum_{\substack{1\le r\le c-1}}\left\Vert\frac{r}{c}\right\Vert^{-1}\ll c\log g,$$
from which it follows clearly that
\begin{align*}
 \sum_{\substack{\left\lfloor \frac{\sqrt{cd}}{2}\right\rfloor\le \ell\le \left\lfloor \frac{cd}{2}\right\rfloor\\ c\nmid\ell~\text{and}~d| \ell}}\int_{-\frac{1}{2cd}}^{\frac{1}{2cd}}\min\left\{d, \left\Vert c\theta+\frac{\ell}{d}\right\Vert^{-1}\right\}&\min\left\{c, \left\Vert d\theta+\frac{\ell}{c}\right\Vert^{-1}\right\}d\theta\ll \log g.
\end{align*}

Case III.  For $\left\lfloor \frac{\sqrt{cd}}{2}\right\rfloor\le \ell\le \left\lfloor \frac{cd}{2}\right\rfloor$ with $c| \ell$ but $d\nmid \ell$, we have
\begin{align*}
\sum_{\substack{\left\lfloor \frac{\sqrt{cd}}{2}\right\rfloor\le \ell\le \left\lfloor \frac{cd}{2}\right\rfloor\\ c|\ell~\text{and}~d\nmid \ell}}\int_{-\frac{1}{2cd}}^{\frac{1}{2cd}}\min\left\{d, \left\Vert c\theta+\frac{\ell}{d}\right\Vert^{-1}\right\}&\min\left\{c, \left\Vert d\theta+\frac{\ell}{c}\right\Vert^{-1}\right\}d\theta\ll \log g.
\end{align*}
via the same argument as Case II. 

Case IV. For $\left\lfloor \frac{\sqrt{cd}}{2}\right\rfloor\le \ell\le \left\lfloor \frac{cd}{2}\right\rfloor$ with $c| \ell$ and $d| \ell$, we have $cd|\ell$ since $(c,d)=1$, which is certainly a contradiction with $\ell\le \left\lfloor \frac{cd}{2}\right\rfloor$.

Gathering together Cases I to IV, we established Eq. (\ref{eq3-2}), and hence Eq. (\ref{eq3-1}). 

This completes the proof of Lemma \ref{lem2}.
\end{proof}

\begin{proof}[Proof of Proposition \ref{proposition1}]
The treatment of the minor arcs benefits from the following trivial estimates
\begin{align*}
\left|\int_{\mathfrak{m}(Q)}f(\alpha)h(-\alpha)d\alpha\right|\le \sup_{\alpha\in \mathfrak{m}(Q)}\left|f(\alpha)\right|\int_{0}^{1}|h(-\alpha)|d\alpha
\end{align*}
together with Lemma \ref{lem1} and  Lemma \ref{lem2}.
\end{proof}

\section{Calculations of the major arcs}
We would provide, in this section, the asymptotic formula of the integral on major arcs as the following proposition.

\begin{proposition}\label{proposition2} For $Q<c^{1/3}$, we have
\begin{align*}
\int_{\mathfrak{M}(Q)}f(\alpha)h(-\alpha)d\alpha=	\frac{g}{2}+O\left(\frac{g}{Q}(\log g)^2+gQ^2\exp\left(-\kappa_1\sqrt{\log g}\right)+dQ^3\right).
\end{align*}
\end{proposition}
To this aim, we firstly prove several lemmas.

\begin{lemma}\label{lem4-1}
For any real number $\theta$, we have
\begin{align*}
f(\theta)=\sum_{0\le n\le g}e(n\theta)+O\left(g(1+|\theta|g)\exp\left(-\kappa_1\sqrt{\log g}\right)\right).
\end{align*}
\end{lemma}
\begin{proof}
Let  $\rho_n=\Lambda(n)-1$. Then for any real number $\theta$, we have
\begin{align}\label{eq4-1}
f(\theta)-\sum_{0\le n\le g}e(n\theta)=\sum_{0\le n\le g}\rho_ne(n\theta).
\end{align}
Integrating by parts, we have
\begin{align}\label{eq4-2}
\sum_{0\le n\le g}\rho_ne(n\theta)=e(g\theta)\sum_{0\le n\le g}\rho_n-2\pi i\theta\int_{0}^{g}\left(\sum_{0\le n\le t}\rho_n\right)e(t\theta)dt.
\end{align}
Employing the elaborate form of the prime number theorem (see e.g. \cite[Lemma 3.1]{Vaughan}), we get 
$$\sum_{0\le n\le t}\rho_n\ll t\exp\left(-\kappa_1\sqrt{\log t}\right),$$
where $\kappa_1>0$ is an absolute constant. Inserting the above estimates into Eqs. (\ref{eq4-1}) and (\ref{eq4-2}), we obtain that for any $\theta$,
\begin{align*}
f(\theta)=\sum_{0\le n\le g}e(n\theta)+O\left(g(1+|\theta|g)\exp\left(-\kappa_1\sqrt{\log g}\right)\right).
\end{align*}

This completes the proof of Lemma \ref{lem4-1}.
\end{proof}

\begin{lemma}\label{lem4-2}
\begin{align*}
\int_{|\theta|\le \frac{Q}{g}}f\left(\theta\right)h\left(-\theta\right)d\theta=\frac{g}{2}+O\left(\frac{g}{Q}(\log g)^2+gQ^2\exp\left(-\kappa_1\sqrt{\log g}\right)\right).
\end{align*}
\end{lemma}
\begin{proof}
From Lemma \ref{lem4-1}, we have
\begin{align*}
\int_{|\theta|\le \frac{Q}{g}}f\left(\theta\right)h\left(-\theta\right)d\theta&=\int_{|\theta|\le \frac{Q}{g}}\sum_{0\le n\le g}e(n\theta)\sum_{\substack{0\le x\le d\\ 0\le y\le c}}e(-\theta (cx+dy))d\theta+\mathcal{R}(\theta),
\end{align*}
where the error term $\mathcal{R}(\theta)$ can be bounded easily by the trivial estimates as
\begin{align*}
\mathcal{R}(\theta)&\ll \int_{|\theta|\le \frac{Q}{g}}g(1+|\theta|g)\exp\left(-\kappa_1\sqrt{\log g}\right)|h(-\theta)|~d\theta\\
&\ll\int_{|\theta|\le \frac{Q}{g}}g(1+|\theta|g)\exp\left(-\kappa_1\sqrt{\log g}\right)g ~d\theta\\
&\ll g^2Q\exp\left(-\kappa_1\sqrt{\log g}\right)\int_{|\theta|\le \frac{Q}{g}}1d\theta\\
&\ll gQ^2\exp\left(-\kappa_1\sqrt{\log g}\right).
\end{align*}
On noting that
$$\int_{|\theta|\le \frac{Q}{g}}=\int_{-1/2}^{1/2}-\int_{\frac{Q}{g}}^{1/2}-\int_{-1/2}^{-\frac{Q}{g}},$$
we can obtain
\begin{align*}
\int_{|\theta|\le \frac{Q}{g}}\sum_{0\le n\le g}e(n\theta)\sum_{\substack{0\le x\le d\\ 0\le y\le c}}e(-\theta (cx+dy))d\theta=\int_{-1/2}^{1/2}\sum_{0\le n\le g}e(n\theta)\sum_{\substack{0\le x\le d\\ 0\le y\le c}}e(-\theta (cx+dy))d\theta\\
+\mathcal{R}_1(\theta)+\mathcal{R}_2(\theta),
\end{align*}
where 
\begin{align*}
\mathcal{R}_1(\theta)=\int_{-1/2}^{-\frac{Q}{g}}\sum_{0\le n\le g}e(n\theta)\sum_{\substack{0\le x\le d\\ 0\le y\le c}}e(-\theta (cx+dy))d\theta
\end{align*}
and 
\begin{align*}
\mathcal{R}_2(\theta)=\int_{\frac{Q}{g}}^{1/2}\sum_{0\le n\le g}e(n\theta)\sum_{\substack{0\le x\le d\\ 0\le y\le c}}e(-\theta (cx+dy))d\theta.
\end{align*}
It follows from the periodic property of $\theta$ that
\begin{align*}
\int_{-1/2}^{1/2}\sum_{0\le n\le g}e(n\theta)\sum_{\substack{0\le x\le d\\ 0\le y\le c}}e(-\theta (cx+dy))d\theta&=\sum_{\substack{0\le n\le g\\0\le x\le d,~ 0\le y\le c}}\int_{0}^{1}e((n-cx-dy)\theta)d\theta.
\end{align*}
Here comes the most interesting point of the whole proof! (The Hardy--Littlewood method reduces the diophantine Frobenius problem with prime variable to the original diophantine Frobenius problem.) By the orthogonality relation, the above integral equals $1$ if $n=cx+dy$ and $0$ otherwise. Hence, by the result of Sylvester \cite{Sylvester} as we mentioned in the introduction (see also \cite[Eq. (2) of the Preface]{RA}), we get
\begin{align*}
\sum_{\substack{0\le n\le g\\0\le x\le d,~ 0\le y\le c}}\int_{0}^{1}e((n-cx-dy)\theta)d\theta=\frac{g+1}{2}.
\end{align*}
For $\frac{Q}{g}\le \theta\le \frac12$, using again the following estimates
$$\sum_{0\le n\le g}e(n\theta)\ll \min\{g,\Vert \theta\Vert^{-1}\}\ll \Vert \theta\Vert^{-1}=\frac1{\theta}$$
and 
$$\sum_{0\le x\le d}e(-\theta cx)\ll\min\{d,\Vert c\theta\Vert^{-1}\}, \quad \sum_{0\le x\le c}e(-\theta dx)\ll\min\{c,\Vert d\theta\Vert^{-1}\},$$
we deduce that
\begin{align*}
\mathcal{R}_2(\theta)&\ll\int_{\frac{Q}{g}}^{1/2}\frac1{\theta}\min\{d,\Vert c\theta\Vert^{-1}\}\min\{c,\Vert d\theta\Vert^{-1}\}d\theta\\
&\le \frac{g}{Q}\int_{\frac{Q}{g}}^{1/2}\min\{d,\Vert c\theta\Vert^{-1}\}\min\{c,\Vert d\theta\Vert^{-1}\}d\theta\\
&\le \frac{g}{Q}(\log g)^2,
\end{align*}
where the last inequality follows from Eq. (\ref{eq3-1}). The same argument would also lead to the estimate $\mathcal{R}_1(\theta)\ll  \frac{g}{Q}(\log g)^2$.
Collecting the estimates above, we obtain that
\begin{align*}
\int_{|\theta|\le \frac{Q}{g}}f\left(\theta\right)h\left(-\theta\right)d\theta=\frac{g}{2}+O\left(\frac{g}{Q}(\log g)^2+gQ^2\exp\left(-\kappa_1\sqrt{\log g}\right)\right).
\end{align*}

This completes the proof of Lemma \ref{lem4-2}.
\end{proof}

\begin{lemma}\label{lem4-3}
For $Q<c^{1/3}$, we have
$$\sum_{2\le q\le Q}\sum_{\substack{1\le a\le q\\ (a,q)=1}}\int_{|\theta|\le \frac{Q}{qg}}f\left(\frac{a}{q}+\theta\right)h\left(-\frac{a}{q}-\theta\right)d\theta\ll dQ^3.$$
\end{lemma}

\begin{proof}
Recall that
\begin{align*}
	h(-\alpha)=\sum_{x\le d}e(-\alpha cx)\sum_{y\le c}e(-\alpha dy)\ll \min\{d, ||c\alpha||^{-1}\}\min\{c, ||d\alpha||^{-1}\},
\end{align*}
hence we have
\begin{align*}
h\left(-\frac{a}{q}-\theta\right)\ll \min\left\{d, \left\Vert c\left(-\frac{a}{q}-\theta\right)\right\Vert^{-1}\right\}\min\left\{c, \left\Vert d\left(-\frac{a}{q}-\theta\right)\right\Vert^{-1}\right\}.
\end{align*}
Since $(a,q)=1$, we would have
\begin{align*}
 \left\Vert c\left(-\frac{a}{q}-\theta\right)\right\Vert\ge \frac{1}{2q} \quad \text{if} \quad q\nmid c
\end{align*}
and
\begin{align*}
\left\Vert d\left(-\frac{a}{q}-\theta\right)\right\Vert\ge \frac{1}{2q} \quad \text{if} \quad q\nmid d
\end{align*}
for $q\ge 2$ and $|\theta|\le \frac{Q}{qg}$, provided that $Q\le c^{1/3}$. Recall that $(c,d)=1$, thus at least one of the above inequalities is admissible. Therefore, for all $2\le q\le Q$, $(a,q)=1$ and $|\theta|\le \frac{Q}{qg}$ we have
\begin{align*}
h\left(-\frac{a}{q}-\theta\right)\ll qd,
\end{align*}
from which we can conclude that
\begin{align*}
\sum_{2\le q\le Q}\sum_{\substack{1\le a\le q\\ (a,q)=1}}\int_{|\theta|\le \frac{Q}{qg}}f\left(\frac{a}{q}+\theta\right)h\left(-\frac{a}{q}-\theta\right)d\theta
\ll\sum_{2\le q\le Q}\sum_{\substack{1\le a\le q\\ (a,q)=1}}\int_{|\theta|\le \frac{Q}{qg}} gdq~ d\theta\ll dQ^3.
\end{align*}

This completes the proof of Lemma \ref{lem4-3}.
\end{proof}

\begin{proof}[Proof of Proposition \ref{proposition2}]
Let's turn back to Eq. (\ref{eq2-4}):
$$\int_{\mathfrak{M}(Q)}f(\alpha)h(-\alpha)d\alpha=\sum_{1\le q\le Q}\sum_{\substack{1\le a\le q\\ (a,q)=1}}\int_{|\theta|\le \frac{Q}{qg}}f\left(\frac{a}{q}+\theta\right)h\left(-\frac{a}{q}-\theta\right)d\theta.$$
The term of the above sum for $a=q=1$ equals 
\begin{align*}
\int_{|\theta|\le \frac{Q}{g}}f\left(1+\theta\right)h\left(-1-\theta\right)d\theta=\int_{|\theta|\le \frac{Q}{g}}f\left(\theta\right)h\left(-\theta\right)d\theta.
\end{align*} 
Now, our proposition follows immediately from Lemma \ref{lem4-2} and lemma \ref{lem4-3}.
\end{proof}

\section{Proofs of Theorem \ref{thm1}, Theorem \ref{thm2} and Corollary \ref{corollary1}}
\begin{proof}[Proof of Theorem \ref{thm2}]
By Eqs. (\ref{eq2-3}), Proposition \ref{proposition1} and Proposition \ref{proposition2}, we have
$$\psi_{c,d}=\frac{g}{2}+O\left(\frac{g}{Q}(\log g)^2+gQ^2\exp\left(-\kappa_1\sqrt{\log g}\right)+dQ^3+\frac{g(\log g)^6}{Q^{1/2}}+g^{4/5}(\log g)^6\right),$$
where $Q\le c^{1/3}$. We choose $Q=(\log g)^{14}$. Then 
\begin{align}\label{eq5-1}
\psi_{c,d}=\frac{g}{2}+O\left(\frac{g}{\log g}\right),
\end{align}
provided that $c\ge (\log g)^{43}$. 

For the complement of our proof, it remains to consider the case that $c\le (\log g)^{43}$.
Following the proof of \cite[Page 299]{Ding}, we have
\begin{align}\label{eq5-2}
	\psi_{c,d}&=\sum_{\substack{n=cx+dy\\n\leqslant g\\x,y\in \mathbb{Z}_{\geqslant0}}}\Lambda(n)\nonumber\\
	&=\sum_{\substack{1\le y\leqslant c\\(y,c)=1}}\sum_{\substack{n\equiv dy\!\!\pmod{c}\\dy\leqslant n\leqslant g}}\Lambda(n)+O\left(1\right)\nonumber\\
	&=\sum_{\substack{1\le y\leqslant c\\(y,c)=1}}\left(\psi(g;c,dy)-\psi(dy;c,dy)\right)+O\left(1\right)\nonumber\\
	&=\psi(g)-\sum_{\substack{1\le y\leqslant c\\(y,c)=1}}\psi(dy;c,dy)+O\left(1\right).
\end{align}
Now, for $c\le (\log g)^{43}\ll (\log d)^{43}$, by the Siegel--Walfisz theorem we have
\begin{align*}
	\sum_{ \substack{1\le y\leqslant c\\(y,c)=1}}\psi(dy;c,dy)&=\sum_{\substack{1\le y\leqslant c\\(y,c)=1}}\left(\frac{dy}{\varphi(c)}+O\left(dy\exp(-\kappa_2\sqrt{\log g})\right)\right)\\
	&=\frac{1}{2}cd+O\left(g\exp\left(-\kappa_3\sqrt{\log g})\right)\right)\\
	&=\frac{1}{2}g+O\left(\frac{g}{c}+g\exp\left(-\kappa_3\sqrt{\log g})\right)\right),
\end{align*}
where $\kappa_2$ and $\kappa_3$ are two positive integers with $\kappa_3<\kappa_2$. Since 
$$\psi(g)=g+O\left(g\exp\left(-\kappa_4\sqrt{\log g})\right)\right)$$
by the prime number theorem, we finally conclude from Eq. (\ref{eq5-2}) that
\begin{align}\label{eq5-3}
		\psi_{c,d}=\frac{1}{2}g+O\left(\frac{g}{c}+g\exp\left(-\kappa_5\sqrt{\log g})\right)\right)
\end{align}	
for $c\le (\log g)^{43}$, where $\kappa_5=\min\{\kappa_3,\kappa_4\}$.

From Eqs. (\ref{eq5-1}) and (\ref{eq5-3}), we proved that
$$\psi_{c,d}\sim \frac12 g, \quad \text{as~}c\rightarrow\infty.$$
This completes the proof of Theorem \ref{thm2}.
\end{proof}

\begin{proof}[Proof of Theorem \ref{thm1}]
For $t\leqslant g$, let
$$\vartheta_{a,b}(t)=\sum_{\substack{p=ax+by\\p\leqslant t\\x,y\in \mathbb{Z}_{\geqslant0}}}\log p \quad \text{and} \quad \vartheta_{a,b}=\vartheta_{a,b}(g).$$
Integrating by parts, we obtain that
\begin{align}\label{E1}
	\pi_{a,b}=\sum_{\substack{p=ax+by\\p\leqslant g\\x,y\in \mathbb{Z}_{\geqslant0}}}1=\frac{\vartheta_{a,b}}{\log g}+\int_{2}^{g}\frac{\vartheta_{a,b}(t)}{t\log ^2t}dt.
\end{align}
 By the Chebyshev estimate, we have
 $$\vartheta_{a,b}(t)\leqslant\sum_{p\leqslant t}\log p\ll t,$$  
 from which it follows that
\begin{equation}\label{E2}
	\int_{2}^{g}\frac{\vartheta_{a,b}(t)}{t\log ^2t}dt\ll\int_{2}^{g}
	\frac{1}{\log^2 t}dt\ll\frac{g}{(\log g)^2}.
\end{equation}
Again, using the Chebyshev estimate, we have
\begin{equation}\label{E3}
	\vartheta_{a,b}=\psi_{a,b}+O(\sqrt{g}).
\end{equation}
Thus, by Theorem \ref{thm2} and Eqs. (\ref{E1}), (\ref{E2}), (\ref{E3}), we conclude that
\begin{equation*}
	\pi_{a,b}=\frac{\psi_{a,b}}{\log g}+O\left(\frac{\sqrt{g}}{\log g}+\frac{g}{(\log g)^2}\right)\sim \frac12 \pi(g),
\end{equation*}
as $c\rightarrow\infty$.
This completes the proof of Theorem \ref{thm1}.
\end{proof}

\begin{proof}[Proof of Corollary \ref{corollary1}]
	It follows clearly from Theorem \ref{thm2}, Eq. (\ref{eq5-2}) and the prime number theorem.
\end{proof}

\section*{Acknowledgments}
The first named author is supported by National Natural Science Foundation of China  (Grant No. 12201544), Natural Science Foundation of Jiangsu Province, China (Grant No. BK20210784), China Postdoctoral Science Foundation (Grant No. 2022M710121), the foundations of the projects ``Jiangsu Provincial Double--Innovation Doctor Program'' (Grant No. JSSCBS20211023) and ``Golden  Phoenix of the Green City--Yang Zhou'' to excellent PhD (Grant No. YZLYJF2020PHD051).

The second named author is support by the National Natural Science Foundation of China (Grant No. 11971476).

The third named author is support by the National Key Research and Development Program of China (Grant No. 2021YFA1000700) and National Natural Science Foundation of China (Grant No. 11922113).

\end{document}